\documentclass[12pt]{article}
\usepackage{amsmath,amsthm}
\usepackage{amssymb,mathrsfs}
\usepackage{bm,mathtools}
\usepackage{booktabs}
\usepackage[blocks]{authblk}
\usepackage{enumitem,float}
\usepackage{tikz}
\usetikzlibrary{positioning, shapes,calc}
\usepackage{caption, subcaption, graphicx,diagbox}
\usepackage{algpseudocode}

\usepackage{xcolor}

\usepackage[a4paper]{geometry}
\usepackage[colorlinks=true,linkcolor=blue,citecolor=blue,urlcolor=blue]{hyperref}
\usepackage[numbers,sort&compress]{natbib}

\makeatletter
\def\@seccntDot{.}
\def\@seccntformat#1{\csname the#1\endcsname\@seccntDot\hskip 0.5em}
\renewcommand\section{\@startsection{section}{1}{\z@}%
{18\p@ \@plus 6\p@ \@minus 3\p@}%
{9\p@ \@plus 6\p@ \@minus 3\p@}%
{\large\bfseries\boldmath}}
\renewcommand\subsection{\@startsection{subsection}{2}{\z@}%
{12\p@ \@plus 6\p@ \@minus 3\p@}%
{3\p@ \@plus 6\p@ \@minus 3\p@}%
{\bfseries\boldmath}}
\renewcommand\subsubsection{\@startsection{subsubsection}{3}{\z@}%
{12\p@ \@plus 6\p@ \@minus 3\p@}%
{\p@}%
{\bfseries\boldmath}}
\makeatother

\theoremstyle{plain}
\newtheorem{theorem}{Theorem}

\newtheorem{corollary}[theorem]{Corollary}

\newtheorem{problem}{Problem}

\theoremstyle{definition}
\newtheorem{definition}{Definition}
\newtheorem{remark}{Remark}

\numberwithin{equation}{section}
\allowdisplaybreaks
\title{On walk domination: Between different types of walks and $m_3$-paths\thanks{This work was supported by the National Natural Science Foundation of China (No. 12331014), Science and Technology Commission of Shanghai Municipality (No. 22DZ2229014), Natural Science Foundation of Fujian Province, China (No. 2024J01875), Science-Technology Foundation for Middle-aged and Young Scientist of Fujian Province (No. JAT220306), and Science-Technology Foundation of Putian University (No. 2023059).} }

\author[a]{Hangdi Chen\thanks{Email: \texttt{chenhangdi188@126.com}}}
\author[b]{Yuhan Ma\thanks{Corresponding author. Email: \texttt{kid4068@163.com}}}
\author[b]{Qingjie Ye\thanks{Email: \texttt{qjye@math.ecnu.edu.cn}}}

\affil[a]{Fujian Key Laboratory of Financial Information Processing, Key Laboratory of Applied Mathematics of Fujian Province University, Putian University, Fujian Putian 351100, China}
\affil[b]{School of Mathematical Sciences, Key Laboratory of MEA (Ministry of Education), Shanghai Key Laboratory of PMMP, East China Normal University, Shanghai 200241, China}

\date{}

\begin{document}

\maketitle

\begin{abstract}
Given two non-adjacent vertices \( u \) and \( v \), we say a $uv$-walk \( W \) dominates a $uv$-walk \( W' \) if every internal vertex of \( W' \) is adjacent to some internal vertex of \( W \) or belongs to \( W \). A class of walks \(\mathbf{A}\) dominates a class of walks \(\mathbf{B}\) if for every pair of non-adjacent vertices $u,v$ in the graph, every $uv$-walk in \(\mathbf{A}\) dominates every $uv$-walk in \(\mathbf{B}\). 
This paper investigates the domination relationships among various types of walks connecting two non-adjacent vertices in a graph. 
In particular, we focus on a problem proposed by Tondato (2024).
We study the domination between different walk types (shortest paths, toll walks, weakly toll walks, $l_k$-paths for $k\in \left\{2,3\right\}$) and $m_3$-paths. Furthermore, we show how these relationships give rise to characterizations of graph classes. 

\bigskip
\noindent\textbf{Keywords.} Walk domination, $m_3$-path, HHD-free.

\noindent\textbf{Mathematics Subject Classification.} 05C38, 05C69
\end{abstract}

\section{Introduction}


Let $G$ be a finite, undirected, simple, and connected graph with vertex set $V(G)$ and edge set $E(G)$. For any pair of vertices \( u, v \in V(G) \), a \textit{\( uv \)-walk} is a sequence of vertices $W: u = v_0, v_1, \ldots, v_n = v$ such that $v_{i-1}$ is adjacent to $v_i$ for all $i \in \{1, \ldots, n\}$. The vertices in the sequence need not be distinct. The vertices \( u \) and \( v \) are the \textit{ends} of the walk, while \( v_1, \ldots, v_{n-1} \) are its \textit{internal vertices}.
The integer \( n \) is the \textit{length of the walk}. We use $W[v_i,v_j]$ ($i\leq j$) to denote the vertices of the subwalk in the walk $W$ between $v_i$ and $v_j$. 

Next, we define several special types of walks discussed in this paper. A \textit{\( uv \)-path} is a \( uv \)-walk with distinct vertices. 
A $uv$-\textit{induced path} (or \textit{monophonic path} \cite{Farber1986}) is a 
path from $u$ to $v$ where two vertices are adjacent if and only if they are consecutive. 
%
%
A \textit{$uv$-shortest path} (or \textit{geodesic} \cite{Farber1986}) is a path between two vertices $u,v$ in a graph $G$ with a length equal to their \textit{distance} $d(u, v)$, which is the minimum number of edges on a connecting path.
A \textit{\( uv \)-\( m_3 \) path} \cite{Dragan1999} is a \( uv \)-induced path of length at least three.
A $uv$-\textit{weakly toll walk} \cite{Dourado2024} is a $uv$-walk such that $u$ is adjacent only to the vertex $v_{1}$, with possibly $\left\{v_{1}\right\} \cap\left\{v_{2}, \ldots, v_{k-1}\right\} \neq \emptyset$, and $v$ is adjacent only to the vertex $v_{k-1}$, with possibly $\left\{v_{k-1}\right\} \cap\left\{v_{1}, \ldots, v_{k-2}\right\} \neq \emptyset$. 
A $uv$-\textit{toll walk} \cite{Alcon2015} is a $uv$-walk satisfying that $u$ is adjacent only to the vertex $v_{1}$, $v$ is adjacent only to the vertex $v_{k-1}$, $\left\{v_{1}\right\} \cap\left\{v_{2}, \ldots, v_{k-1}\right\}=\emptyset$ and $\left\{v_{k-1}\right\} \cap \left\{v_{1}, \ldots, v_{k-2}\right\}=\emptyset$.
A $uv$-$l_k$-\textit{path} is a $uv$-induced path with length at most $k$. 
%
%
We use the following notation for the sets of different walk types between two non-adjacent vertices \( u \) and \( v \):
\begin{align*}
\mathbf{SP}(u, v) &= \{ W : W \text{ is a } uv\text{-shortest path} \}, \\
\mathbf{IP}(u, v) &= \{ W : W \text{ is a } uv\text{-induced path} \}, \\
\mathbf{P}(u, v) &= \{ W : W \text{ is a } uv\text{-path} \}, \\
\mathbf{m_3}(u, v) &=\{W: W \text{ is a }uv\text{-}m_3 \text{ path} \}, \\
\mathbf{TW}(u, v) &= \{ W : W \text{ is a } uv\text{-toll walk} \}, \\
\mathbf{WTW}(u, v) &= \{ W : W \text{ is a } uv\text{-weakly toll walk} \}, \\
\mathbf{W}(u,v) &=\{W: W \text{ is a } uv\text{-walk} \}\\
\mathbf{l_k}(u, v) &= \{ W : W \text{ is a } uv\text{-}l_k\text{-path} \}.
\end{align*}
These sets of walks exhibit several inclusion relationships. For instance, a key property of shortest paths is that they are always induced paths, which means $\mathbf{SP}(u, v) \subseteq \mathbf{IP}(u, v)$.
The following remark summarizes the relationships between the different types of walks we have considered, which follow directly from the definitions.
\begin{remark}\label{rem1}
For any non-adjacent vertices $u,v$,
\begin{equation*}
\left.
\begin{array}{r}
    \mathbf{SP}(u, v) \\
    \mathbf{m_3}(u, v) \\
    \mathbf{l_2}(u, v) \subseteq \mathbf{l_3}(u, v)
\end{array}
\right\}
\subseteq \mathbf{IP}(u, v) \subseteq
\left\{
\begin{array}{l}
    \mathbf{P}(u, v) \\
    \mathbf{TW}(u, v) \subseteq \mathbf{WTW}(u, v)
\end{array}
\right\}
\subseteq \mathbf{W}(u, v).
\end{equation*}
\end{remark}





A powerful framework for analyzing the interaction between different types of walks is that of \textit{walk domination}.
\begin{definition}\label{def1}
    Let \( u \) and \( v \) be two non-adjacent vertices in $G$. The \( uv \)-walk \( W : u, v_1, \ldots, v_{m-1}, v \) \emph{dominates} the \( uv \)-walk \( W' : u, v'_1, \ldots, v'_{n-1}, v \) if every internal vertex of \( W' \) is adjacent to some internal vertex of \( W \) or belongs to \( W \).
\end{definition}

\begin{definition}\label{def2}
Let $\mathbf{A/B}$ be the class of graphs $G$ such that for every pair of non-adjacent vertices $u$ and $v$ of $G$, every walk $W \in \mathbf{A}(u, v)$ dominates every walk $W' \in \mathbf{B}(u, v)$ (which is vacuously true if $\mathbf{A}(u, v) = \emptyset$).
\end{definition}



For any given class of graphs, a natural question is whether certain types of walks consistently dominate others. The study of walk domination not only explores this question but also seeks to characterize classes of graphs based on this very property.
The following definitions for specific graph classes are fundamental to our results.

A \textit{cycle} $C$ of length $n$ in a graph $G$ is a path $P: v_1,\ldots,v_n$ plus an edge between $v_1$ and $v_n$. Each edge of $G$ between two non-consecutive vertices of $C$ is called a \textit{chord}. The cycle of length $n$ without chords is denoted by $C_n$.
A \textit{hole} is a chordless cycle with at least five vertices.
A \textit{house} is the complement of an induced path with five vertices.
A \textit{domino} (or $D$) is the graph obtained from the chordless cycle $x_0,x_1\ldots,x_5$ by adding the chord $x_1x_4$. Some other graphs used in this paper are shown in Figure~\ref{fig:graph}.
%

Over the years, extensive work has been done concerning the domination number and
its variations. Walks in graphs are subgraphs that inform us about the topological structure of graphs. Hence, as a variation of the domination number, walk domination has been widely discussed. In \cite{Alcon2016,Farber1987,GM,TSB} it was proved that the notion of domination between different types of walks plays a central role in characterizations of graph classes. Moreover, the walks studied in \cite{Alcon2016,Alcon2015,Changat1999,Changat2001,Changat2010,Dirac1961,Dragan1999,Dourado2024,Farber1986,Farber1987,GM,TSB,counter} are related to convexities defined over a walk system. Standard graph classes like interval and superfragile \cite{GM} have been characterized. It should be noted that some graph classes characterized by walk domination are not hereditary \cite{Alcon2016,GM}, i.e. they can not be characterized as \(\mathcal{F}\)-free. Note that we say that a graph $G$ is $\mathcal{F}$-free if $G$ does not contain any induced subgraph that belongs to $\mathcal{F}$.


Some important classes of graphs have been characterized by domination between different types of walks. In \cite{Alcon2016,TSB}, Alc$\acute{\text{o}}$n and Tondato considered walks, tolled-walks, paths, induced paths, or shortest paths. Table \ref{table1}, with walk types $\mathbf{A}$ in the first column and $\mathbf{B}$ in the first row, describes the graph classes $\mathbf{A/B}$ that result from these domination relationships. The following theorem summarizes some of the key results from the table:

\begin{theorem}[\citet{TSB}]\label{thm:hhd}
      $\mathbf{IP}/\mathbf{m_3}=\mathbf{W}/\mathbf{m_3}=\text{HHD-free}$.
\end{theorem}
\begin{theorem}[\citet{TSB}]\label{thm:m3w}
    $\mathbf{m_3}/\mathbf{W}=\left\{P_4,A, \overline{\text{gem}\cup K_2}, C_5, \overline{X_{58}},X_{96}, F_3 \right\}$-$\text{free}$.
\end{theorem}
\begin{theorem}[\citet{TSB}]\label{thm:m3ip}
   $\mathbf{m_3}/\mathbf{IP}=\left\{\text{hole}, D, \text{Antenna}, X_5 \right\}$-$\text{free}$.
\end{theorem}

\begin{figure}[p]
    \centering
     \begin{subfigure}{0.3\textwidth}
        \centering
   \begin{tikzpicture}
 \tikzstyle{every node}=[circle, fill=black, inner sep=2pt, draw=none];
 \node (x0) at (-2,0)  [label=left:$x_0$] {};
 \node (x1) at (-1,0)  [label=above:$x_1$] {};
 \node (x2) at (0,0)  [label=above:$x_2$] {};
 \node (x3) at (1,0)  [label=above:$x_3$] {};
 \node (x4) at (2,0)  [label=right:$x_4$] {};
 \draw [thick] (x0) -- (x1);
 \draw [thick] (x1) -- (x2);
 \draw [thick] (x2) -- (x3);
 \draw [thick] (x3) -- (x4);
\end{tikzpicture}
\caption{$P_4$}
\end{subfigure}
\hfill
 \begin{subfigure}{0.3\textwidth}
        \centering
   \begin{tikzpicture}
 \tikzstyle{every node}=[circle, fill=black, inner sep=2pt, draw=none];
 \node (x0) at (-1.5,0)  [label=left:$x_0$] {};
 \node (x1) at (-0.5,0)  [label=below:$x_1$] {};
 \node (x2) at (0.5,0)  [label=below:$x_2$] {};
 \node (x3) at (1.5,0)  [label=right:$x_3$] {};
 \node (x4) at (0,1)  [label=right:$x_4$] {};
 \node (x5) at (0,2)  [label=above:$x_5$] {};
 \draw [thick] (x0) -- (x1);
 \draw [thick] (x1) -- (x2);
 \draw [thick] (x2) -- (x3);
 \draw [thick] (x3) -- (x4);
 \draw [thick] (x0) -- (x4);
 \draw [thick] (x1) -- (x4);
 \draw [thick] (x2) -- (x4);
 \draw [thick] (x4) -- (x5);
\end{tikzpicture}
\caption{$\overline{gem\cup K_2}$}
\end{subfigure}
\hfill
 \begin{subfigure}{0.3\textwidth}
        \centering
   \begin{tikzpicture}
 \tikzstyle{every node}=[circle, fill=black, inner sep=2pt, draw=none];
 \node (x0) at (0.5,1.5)  [label=above:$x_0$] {};
 \node (x1) at (1,1)  [label=right:$x_1$] {};
 \node (x2) at (1,0)  [label=below:$x_2$] {};
 \node (x3) at (0,0)  [label=left:$x_3$] {};
 \node (x4) at (0,1)  [label=left:$x_4$] {};
 \draw [thick] (x0) -- (x1);
 \draw [thick] (x1) -- (x2);
 \draw [thick] (x2) -- (x3);
 \draw [thick] (x3) -- (x4);
 \draw[densely dashed] (x4) -- (x0);
\end{tikzpicture}
\caption{Hole}\label{fig:hole}
\end{subfigure}
\hfill
    \begin{subfigure}{0.3\textwidth}
        \centering
   \begin{tikzpicture}
 \tikzstyle{every node}=[circle, fill=black, inner sep=2pt, draw=none];
 \node (x0) at (0,2)  [label=left:$x_0$] {};
 \node (x1) at (0,1)  [label=left:$x_1$] {};
 \node (x2) at (0,0)  [label=left:$x_2$] {};
 \node (x3) at (1,0)  [label=right:$x_3$] {};
 \node (x4) at (1,1)  [label=right:$x_4$] {};
 \node (x5) at (1,2)  [label=right:$x_5$] {};
 
 \draw [thick] (x0) -- (x1);
 \draw [thick] (x0) -- (x5);
 \draw [thick] (x1) -- (x2);
 \draw [thick] (x1) -- (x4);
 \draw [thick] (x2) -- (x3);
 \draw [thick] (x3) -- (x4);
 \draw [thick] (x4) -- (x5);
\end{tikzpicture}
\caption{D (Domino)}\label{fig:domino}
\end{subfigure}
\hfill
   \begin{subfigure}{0.3\textwidth}
        \centering
   \begin{tikzpicture}
 \tikzstyle{every node}=[circle, fill=black, inner sep=2pt, draw=none];
 \node (x0) at (-1.5,0)  [label=left:$x_0$] {};
 \node (x1) at (-0.5,0)  [label=below:$x_1$] {};
 \node (x2) at (0.5,0)  [label=below:$x_2$] {};
 \node (x3) at (1.5,0)  [label=right:$x_3$] {};
 \node (x4) at (0,1)  [label=right:$x_4$] {};
 \node (x5) at (0,2)  [label=right:$x_5$] {};
 
 \draw [thick] (x0) -- (x1);
 \draw [thick] (x1) -- (x2);
 \draw [thick] (x1) -- (x4);
 \draw [thick] (x2) -- (x3);
 \draw [thick] (x2) -- (x4);
 \draw [thick] (x4) -- (x5);
\end{tikzpicture}
\caption{$F_3$}
\end{subfigure}
\hfill 
   \begin{subfigure}{0.3\textwidth}
        \centering
   \begin{tikzpicture}
 \tikzstyle{every node}=[circle, fill=black, inner sep=2pt, draw=none];
 \node (x0) at (1,0)  [label=right:$x_0$] {};
 \node (x1) at (1,1)  [label=right:$x_1$] {};
 \node (x2) at (1,2)  [label=right:$x_2$] {};
 \node (x3) at (0,2)  [label=left:$x_3$] {};
 \node (x4) at (0,1)  [label=left:$x_4$] {};
 \node (x5) at (0,0)  [label=left:$x_5$] {};
 
 \draw [thick] (x0) -- (x1);
 \draw [thick] (x1) -- (x2);
 \draw [thick] (x1) -- (x4);
 \draw [thick] (x2) -- (x3);
 \draw [thick] (x3) -- (x4);
 \draw [thick] (x4) -- (x5);
\end{tikzpicture}
\caption{A}
\end{subfigure}
\hfill
   \begin{subfigure}{0.3\textwidth}
        \centering
   \begin{tikzpicture}
 \tikzstyle{every node}=[circle, fill=black, inner sep=2pt, draw=none];
 \node (x0) at (0,0)  [label=left:$x_0$] {};
 \node (x1) at (0,1)  [label=left:$x_1$] {};
 \node (x2) at (0,2)  [label=left:$x_2$] {};
 \node (x3) at (1,2)  [label=right:$x_3$] {};
 \node (x4) at (1,1)  [label=right:$x_4$] {};
 \node (x5) at (1,0)  [label=right:$x_5$] {};
 
 \draw [thick] (x0) -- (x1);
 \draw [thick] (x1) -- (x2);
 \draw [thick] (x1) -- (x4);
 \draw [thick] (x2) -- (x3);
 \draw [thick] (x3) -- (x4);
 \draw [thick] (x4) -- (x5);
 \draw [thick] (x1) -- (x3);
\end{tikzpicture}
\caption{$X_{96}$}
\end{subfigure}
\hfill 
   \begin{subfigure}{0.3\textwidth}
        \centering
   \begin{tikzpicture}
 \tikzstyle{every node}=[circle, fill=black, inner sep=2pt, draw=none];
 \node (x0) at (0.5,2)  [label=above:$x_0$] {};
 \node (x1) at (0,1)  [label=left:$x_1$] {};
 \node (x2) at (0,0)  [label=left:$x_2$] {};
 \node (x3) at (1,0)  [label=right:$x_3$] {};
 \node (x4) at (1,1)  [label=right:$x_4$] {};
 
 \draw [thick] (x0) -- (x1);
 \draw [thick] (x0) -- (x4);
 \draw [thick] (x1) -- (x2);
 \draw [thick] (x1) -- (x4);
 \draw [thick] (x2) -- (x3);
 \draw [thick] (x3) -- (x4);
\end{tikzpicture}
\caption{House}\label{fig:house}
\end{subfigure}
\hfill
   \begin{subfigure}{0.3\textwidth}
        \centering
   \begin{tikzpicture}
 \tikzstyle{every node}=[circle, fill=black, inner sep=2pt, draw=none];
 \node (x0) at (-0.5,2)  [label=left:$x_0$] {};
 \node (x1) at (0.5,2)  [label=right:$x_1$] {};
 \node (x2) at (1,1)  [label=right:$x_2$] {};
 \node (x3) at (1,0)  [label=right:$x_3$] {};
 \node (x4) at (0,0)  [label=left:$x_4$] {};
 \node (x5) at (0,1)  [label=left:$x_5$] {};
 
 \draw [thick] (x0) -- (x1);
 \draw [thick] (x0) -- (x5);
 \draw [thick] (x1) -- (x2);
 \draw [thick] (x1) -- (x5);
 \draw [thick] (x2) -- (x3);
 \draw [thick] (x2) -- (x5);
 \draw [thick] (x3) -- (x4);
 \draw [thick] (x4) -- (x5);
\end{tikzpicture}
\caption{$X_5$}\label{fig:x5}
\end{subfigure}
\hfill 
 \begin{subfigure}{0.3\textwidth}
        \centering
   \begin{tikzpicture}
 \tikzstyle{every node}=[circle, fill=black, inner sep=2pt, draw=none];
 \node (x0) at (0.5,2)  [label=above:$x_0$] {};
 \node (x1) at (0,1)  [label=left:$x_1$] {};
 \node (x2) at (0,0)  [label=left:$x_2$] {};
 \node (x3) at (1,0)  [label=right:$x_3$] {};
 \node (x4) at (1,1)  [label=right:$x_4$] {};
  \node (x5) at (1.5,2)  [label=right:$x_5$] {};
 \draw [thick] (x0) -- (x1);
 \draw [thick] (x0) -- (x4);
 \draw [thick] (x1) -- (x2);
 \draw [thick] (x1) -- (x4);
 \draw [thick] (x2) -- (x3);
 \draw [thick] (x3) -- (x4);
 \draw [thick] (x4) -- (x5);
\end{tikzpicture}
\caption{$\overline{X_{58}}$}
\end{subfigure}
\hfill 
\begin{subfigure}{0.3\textwidth}
        \centering
   \begin{tikzpicture}
 \tikzstyle{every node}=[circle, fill=black, inner sep=2pt, draw=none];
 \node (x0) at (0.5,2.5)  [label=above:$x_0$] {};
 \node (x1) at (0.5,2)  [label=right:$x_1$] {};
 \node (x2) at (1,1)  [label=right:$x_2$] {};
 \node (x3) at (1,0)  [label=right:$x_3$] {};
 \node (x4) at (0,0)  [label=left:$x_4$] {};
 \node (x5) at (0,1)  [label=left:$x_5$] {};
 
 \draw [thick] (x0) -- (x1);
 \draw [thick] (x1) -- (x2);
 \draw [thick] (x1) -- (x5);
 \draw [thick] (x2) -- (x3);
 \draw [thick] (x2) -- (x5);
 \draw [thick] (x3) -- (x4);
 \draw [thick] (x4) -- (x5);
\end{tikzpicture}
\caption{Antenna}
\end{subfigure}
\hfill
\begin{subfigure}{0.3\textwidth}
        \centering
   \begin{tikzpicture}
 \tikzstyle{every node}=[circle, fill=black, inner sep=2pt, draw=none];
 \node (x0) at (1,2)  [label=above:$x_0$] {};
 \node (x1) at (1.5,1.5)  [label=right:$x_1$] {};
 \node (x2) at (1,1)  [label=right:$x_2$] {};
 \node (x3) at (1,0)  [label=right:$x_3$] {};
 \node (x4) at (0,0)  [label=left:$x_4$] {};
 \node (x5) at (0,1)  [label=left:$x_5$] {};
 \node (x6) at (0.5,1.5)  [label=left:$x_6$] {};
 
 \draw [thick] (x0) -- (x1);
 \draw [thick] (x0) -- (x6);
 \draw [thick] (x1) -- (x2);
 \draw [thick] (x5) -- (x6);
 \draw [thick] (x2) -- (x6);
 \draw [thick] (x2) -- (x3);
 \draw [thick] (x2) -- (x5);
 \draw [thick] (x3) -- (x4);
 \draw [thick] (x4) -- (x5);
\end{tikzpicture}
\caption{$F$}\label{fig:f},
\end{subfigure}
\hfill 
\caption{Graphs used to describe the graph classes}\label{fig:graph}
\end{figure}

\begin{table}[htp]
    \centering
\begin{tabular}{c|c c c c c}
        \toprule
        \diagbox{$\mathbf{A}$}{$\mathbf{B}$}& $\mathbf{SP}$ & $\mathbf{IP}$ & $\mathbf{P}$ & $\mathbf{W}$ & $\mathbf{m_3}$  \\
        \hline
        $\mathbf{SP}$ & $\mathbf{g}$-$\mathbf{Ch}$ \cite{Alcon2016} & $\mathbf{Ch}$ \cite{Alcon2016} & $\mathbf{Pt^-}$ \cite{Alcon2016} & $\mathbf{Sup}$ \cite{Alcon2016} &   \\
        $\mathbf{IP}$ & $\mathbf{Ch}$ \cite{Alcon2016} & $\mathbf{Ch}$ \cite{Alcon2016} & $\mathbf{Pt^-}$ \cite{Alcon2016} &  $\mathbf{Sup}$ \cite{Alcon2016} & $\text{HHD-free}$ \cite{TSB}  \\
        $\mathbf{P}$ & $\mathbf{Ch}$ \cite{Alcon2016} & $\mathbf{Ch}$ \cite{Alcon2016} & $\mathbf{Pt^-}$ \cite{Alcon2016} &  $\mathbf{Sup}$ \cite{Alcon2016} & $\text{HHD-free}$ \cite{TSB}  \\
        $\mathbf{TW}$ & $\mathbf{Ch}$ \cite{Alcon2016} & $\mathbf{Ch}$ \cite{Alcon2016} & $\mathbf{Pt^-}$ \cite{Alcon2016} &  $\mathbf{Sup}$ \cite{Alcon2016} &  \\
        $\mathbf{W}$ & $\mathbf{Ch}$ \cite{Alcon2016} & $\mathbf{Ch}$ \cite{Alcon2016} & $\mathbf{Pt^-}$ \cite{Alcon2016} &  $\mathbf{Sup}$ \cite{Alcon2016} & $\text{HHD-free}$ \cite{TSB}  \\
        $\mathbf{m_3}$ &  & $(a)$ \cite{TSB} & $(a)$ \cite{TSB} & $(b)$ \cite{TSB}  \\
        $\mathbf{WTW}$ & $\mathbf{Ch}$ \cite{Alcon2016} & $\mathbf{Ch}$ \cite{Alcon2016} & $\mathbf{Pt^-}$ \cite{Alcon2016} & $\mathbf{Sup}$ \cite{Alcon2016} &  \\
        \bottomrule
    \end{tabular}
     \caption{We denote by $\mathbf{Ch}$ the class of chordal graphs, by $\mathbf{Int}$ the class of interval graphs, by $\mathbf{Sup}$ the class of superfragile graphs, and by $\mathbf{Pt^-}$ the class $\mathbf{Ptolematic^-}$. Here, $(a)=\left\{\text{hole}, D, Antenna, X_5 \right\}$-free, $(b)=\left\{P_4,A, \overline{gem\cup K_2}, C_5, \overline{X_{58}},X_{96}, F_3 \right\}$-free. The definitions of these classes can be found in \cite{TSB}.} 
     \label{table1}
\end{table}

Naturally, Tondato raised a problem after the characterization of $\text{HHD-free}$ graphs in \cite{TSB}, which is defined as the class of graphs containing no house, hole, or domino (see Figure~\ref{fig:graph}) as induced subgraphs. The problem is stated as follows:
\begin{problem}[\citet{TSB}]\label{prob:tsb}
    Do $\mathbf{A}/\mathbf{m_3}$ and $\mathbf{m_3}/\mathbf{A}$, for $\mathbf{A}\in \left\{\mathbf{l_k},\mathbf{SP},\mathbf{TW},\mathbf{WTW} \right\}$ give rise to characterize class of graphs? 
\end{problem}

In this paper, we mainly study the domination between different walk types and $m_3$-paths, and show how these give rise to characterizations of graph classes. 
The main results and conclusions of the above problem are listed and proved in Sections \ref{Result1} and \ref{Result2}. We characterize the classes of graphs of $\mathbf{A}/\mathbf{m_3}$ and part of $\mathbf{m_3}/\mathbf{A}$, for $\mathbf{A}\in \left\{\mathbf{l_k},\mathbf{SP},\mathbf{TW},\mathbf{WTW} \right\}$. The topics that may be motivating for future work are discussed in Section \ref{Con}.

\section{Classes $\mathbf{A}/\mathbf{m_3}$ for $\mathbf{A}\in \left\{\mathbf{l}_2,\mathbf{l}_3,\mathbf{SP},\mathbf{TW},\mathbf{WTW} \right\}$}\label{Result1}
In this section, we will prove that $\mathbf{A}/\mathbf{m_3}=\text{HHD-free}$ for $\mathbf{A}\in \{\mathbf{l}_2,\mathbf{l}_3,\mathbf{SP},\mathbf{TW}$, $\mathbf{WTW} \}$.
%
%
%
First, since $\mathbf{m_3}(u, v) \subseteq \mathbf{IP}(u, v) \subseteq \mathbf{TW}(u, v) \subseteq \mathbf{WTW}(u, v) \subseteq \mathbf{W}(u, v)$ by Remark \ref{rem1}, it follows that $\mathbf{W}/\mathbf{m_3}\subseteq \mathbf{WTW}/\mathbf{m_3}\subseteq \mathbf{TW}/\mathbf{m_3}\subseteq \mathbf{IP}/\mathbf{m_3}$.
By Theorem \ref{thm:hhd}, $\mathbf{IP}/\mathbf{m_3}=\mathbf{W}/\mathbf{m_3}=\text{HHD-free}$, which immediately yields the following corollary.
\begin{corollary}
    $\mathbf{WTW}/\mathbf{m_3}=\mathbf{TW}/\mathbf{m_3}=\text{HHD-free}$.
\end{corollary}

Now, we consider the remaining classes.
\begin{theorem}
$\mathbf{l_2}/\mathbf{m_3}=\mathbf{l_3}/\mathbf{m_3}=\mathbf{SP}/\mathbf{m_3}=\text{HHD-free}$.
\end{theorem}

\begin{proof}
    By Theorem \ref{thm:hhd} and Remark \ref{rem1}, we have $\text{HHD-free}=\mathbf{W}/\mathbf{m_3}\subseteq \mathbf{l_3}/\mathbf{m_3}\subseteq \mathbf{l_2}/\mathbf{m_3}$. Now we only need to prove $\mathbf{l_2}/\mathbf{m_3}\subseteq $ $\text{HHD-free}$.
    
    As shown in Figure \ref{fig:house}, the house has a pair of non-adjacent vertices $u,v$ and a $uv$-$\mathbf{m_3}$ path: $u=x_0,x_4,x_3,x_2=v$ which is not dominated by the $uv$-$l_2$-path: $u=x_0,x_1,x_2=v$ ($x_3$ in the $uv$-$m_3$ path is not dominated). 
    
    As shown in Figure \ref{fig:hole}, a hole has a pair of non-adjacent vertices $u,v$ and a $uv$-$\mathbf{m_3}$ path: $u=x_0,x_n,x_{n-1},\ldots,x_2=v$ which is not dominated by the $uv$-$l_2$-path: $u=x_0,x_1,x_2=v$ ($x_n$ in the $uv$-$m_3$ path is not dominated). 
    
    As shown in Figure \ref{fig:domino}, the domino ($D$) has a pair of non-adjacent vertices $u,v$ and a $uv$-$\mathbf{m_3}$ path: $u=x_0,x_5,x_4,x_3,x_2=v$ which is not dominated by the $uv$-$l_2$-path: $u=x_0,x_1,x_2=v$ ($x_3$ in the $uv$-$m_3$ path is not dominated).

    Hence, we have $\mathbf{l_2}/\mathbf{m_3}=\mathbf{l_3}/\mathbf{m_3}=\text{HHD-free}$.
    
    On the other hand, by Theorem \ref{thm:hhd} and Remark \ref{rem1}, we have $\text{HHD-free}=\mathbf{IP}/\mathbf{m_3}\subseteq \mathbf{SP}/\mathbf{m_3}$. Note that a $uv$-$l_2$-path must be a $uv$-shortest path since we only consider pairs of non-adjacent vertices. Hence $\mathbf{SP}/\mathbf{m_3} \subseteq \mathbf{l_2}/\mathbf{m_3}=\text{HHD-free}$, and thus we can get $\mathbf{SP}/\mathbf{m_3}=\text{HHD-free}$.
%
\end{proof}


\section{Classes $\mathbf{m_3}/\mathbf{A}$ for $\mathbf{A}\in \left\{\mathbf{l_k},\mathbf{SP},\mathbf{TW},\mathbf{WTW} \right\}$}\label{Result2}
In this section we will study dominations between $m_3$-paths and different types of walks like shortest paths, toll walks, weakly toll walks and $l_k$-paths. As a consequence of Remark \ref{rem1}, $\mathbf{m_3}/\mathbf{W}\subseteq \mathbf{m_3}/\mathbf{WTW}\subseteq \mathbf{m_3}/\mathbf{TW}\subseteq \mathbf{m_3}/\mathbf{IP}\subseteq \mathbf{m_3}/\mathbf{SP}$. 
It is easy to see the following conclusion which we can get by combining Theorems \ref{thm:m3w}--\ref{thm:m3ip}.
\begin{theorem}\label{thm:11}
    $\left\{P_4,A, \overline{\text{gem}\cup K_2}, C_5, \overline{X_{58}},X_{96}, F_3 \right\}$-$\text{free}=\mathbf{m_3/W}\subseteq \mathbf{m_3}/\mathbf{WTW}\subseteq \mathbf{m_3}/\mathbf{TW} \subseteq \mathbf{m_3/IP}=\left\{\text{hole}, D, \text{Antenna}, X_5 \right\}$-$\text{free}\subseteq \mathbf{m_3}/\mathbf{SP}$.
\end{theorem}








By Theorem~\ref{thm:11}, we know that if a graph is $\{\text{hole}, D, \text{Antenna}, X_5\}$-free, then it belongs to $\mathbf{m_3}/\mathbf{SP}$. In fact, we can do better. Notice that the Antenna graph is an induced subgraph of $F$. Therefore, by replacing Antenna with $F$, we can obtain an improved sufficient condition.

\begin{theorem}\label{thm:suff}
    If $G$ is $\left\{\text{hole},D,X_5,F \right\}$-\text{free}, then $G\in \mathbf{m_3}/\mathbf{SP}$.
\end{theorem}
\begin{proof}
    Let $G$ be a connected simple $\left\{\text{hole},D,X_5,F \right\}$-free graph and suppose $G\notin \mathbf{m_3}/\mathbf{SP}$. Then there exist two non-adjacent vertices $u$ and $v$, a $uv$-$m_3$ path $W: u=x_0,\ldots,x_n=v$ $(n\ge3)$ and a $uv$-shortest path $W': u=x_0',\ldots,x_h'=v$ satisfying that $W$ does not dominate $W'$. Thus, there is some internal vertex of $W'$ that is neither a vertex of $W$ nor adjacent to any internal vertex of $W$.

    Let $k$ be the first index such that $x_k'$ is neither a vertex of $W$ nor adjacent to any internal vertex of $W$. We consider the following cases.

    \textbf{Case 1.} Suppose $k=1$ (by symmetry $k=h-1$). We have $x_2'\not\in W$. 
    We observe that $u$ is not adjacent to $x_2'$ since $W'$ is a $uv$-shortest path. Let us consider two cases depending on whether $x_1$ is adjacent to $x_2'$.

    \textbf{Case 1.1.} Assume that $x_1$ is not adjacent to $x_2'$. Let $p$ and $q$ be the first indices such that $1\leq p\leq n$, $2\leq q\leq h$, and $x_p$ is adjacent to $x_q'$ or $x_p=x_q'$. Clearly $G[W[u,x_p]\cup W'[u,x_q']]$ is a hole, a contradiction.

    \textbf{Case 1.2.} Suppose that $x_1$ is adjacent to $x_2'$. Note that $G[\left\{u,x_1,x_1',x_2' \right\}]\cong C_4$. Since $W$ is a $uv$-$m_3$ path, $x_2'\neq v$.
    Thus, $h\ge 3$.
    Note that $x_3'\neq x_2$ since $W'$ is a shortest $uv$-path. Since $x_2'\not\in W$, we obtain $x_2\neq x_2'$.

    \textbf{Case 1.2.1.} Suppose that $x_2$ is adjacent to $x_2'$. Now, $G[\left\{u,x_1,x_2,x_1',x_2' \right\}]$ is a house. 
    Note that if $x_3$ is adjacent to $x_2'$, then $G[\left\{u, x_1,x_2,x_3,x_1',x_2' \right\}]\cong X_5$, a contradiction. Hence, $x_3$ is not adjacent to $x_2'$ and $x_3'\neq x_3$. 
    
    Observe that $x_3'\neq x_2$ since $W'$ is a $uv$-shortest path. If $x_3'$ is adjacent to $x_2$, then we get an induced $X_5$, a contradiction. Hence $x_3'$ is not adjacent to $x_2$.

    Note that if $x_3$ is adjacent to $x_3'$, then $G[\left\{u,x_1,x_2,x_3,x_1',x_2',x_3' \right\}]\cong F$, hence $x_3x_3'\not\in E(G)$. Then, let $p$ and $q$ be the first indices such that $3\leq p\leq n$, $3\leq q\leq h$, and $x_p$ is adjacent to $x_q'$ or $x_p=x_q'$. Clearly $G[W[x_2,x_p]\cup W'[x_2',x_q']]$ is a hole or $G[\left\{u,x_1,x_2,x_3,x_4,x_1',x_2' \right\}]\cong F$ or $G[\left\{u,x_1,x_2,x_3',x_4',x_1',x_2' \right\}]\cong F$, a contradiction.

    \textbf{Case 1.2.2.} Suppose $x_2$ is not adjacent to $x_2'$. Now $x_3'\neq x_1,x_2$ and $x_3'$ is not adjacent to $x_1$ since $W'$ is a $uv$-shortest path. We observe that $x_3'$ is not adjacent to $x_2$ since otherwise $G[\left\{u, x_1,x_2,x_3',x_1',x_2' \right\}]\cong D$. But now let $p$ and $q$ be the first indices such that $2\leq p\leq n$, $3\leq q\leq h$, and $x_p$ is adjacent to $x_q'$ or $x_p=x_q'$. Clearly $G[W[x_1,x_p]\cup W'[x_2',x_q']]$ is a hole, a contradiction.

    \textbf{Case 2.} Suppose $k\neq 1,h-1$. Now $x_{k-1}', x_{k+1}'\not\in W$. By the choice of $k$, let $i$ be the last index such that $x_{k-1}'$ is adjacent to $x_i$. Note that $i\neq n$ and $k+1\neq h$. 
    Let us consider two cases depending on whether $x_{k+1}'$ is adjacent to a vertex of $W[x_i,v]$.

    \textbf{Case 2.1.} Suppose $x_{k+1}'$ is not adjacent to any vertex of $W[x_i,v]$. Let $p$ and $q$ be the first indices such that $i\leq p\leq n$, $k+2\leq q\leq h$, and $x_p$ is adjacent to $x_q'$ or $x_p=x_q'$. Clearly $G[W[x_i,x_p]\cup W'[x_{k-1}',x_q']]$ is a hole, a contradiction.

    \textbf{Case 2.2.} Assume that $x_{k+1}'$ is adjacent to some vertex of $W[x_i,v]$. Let $j$ be the first index such that $x_{k+1}'$ is adjacent to $x_j$. If $j>i$, then $G[W[x_i,x_j]\cup W'[x_{k-1}',x_{k+1}']]$ is a hole, a contradiction. Hence $i=j$ and now $G[\left\{x_i,x_{k-1}',x_k',x_{k+1}' \right\}]\cong C_4$.

     \textbf{Case 2.2.1.} Suppose $i=n-1$. As $k\neq h-1$, there exists $x_{k+2}'$ which may be equal to $v$. In fact, whether $x_{k+2}'$ is equal to $v$, $W'[u,x_{k-1}']\cup \left\{x_{n-1},v \right\}$ is a $uv$-path shorter than $W'$, which is impossible.

    \textbf{Case 2.2.2.} Suppose that $i<n-1$. Then there exist $x_{i+1}$ and $x_{i+2}$ in $W$. Note that $x_{i+2}$ may be $v$.

    \textbf{Case 2.2.2.1.} First, assume that $x_{k+1}'$ is adjacent to $x_{i+1}$ and $x_{i+2}$. Then $G[\left\{x_i,x_{i+1},x_{i+2},x_{k-1}',x_k',x_{k+1}' \right\}]\cong X_5$, a contradiction.

    \textbf{Case 2.2.2.2.} Now suppose $x_{k+1}'$ is adjacent to $x_{i+1}$ but it is not adjacent to $x_{i+2}$. We observe that $x_{k+2}'\neq x_{i+2}$. Since $G$ contains no $X_5$, $x_{k+2}'$ is not adjacent to $x_{i+1}$. And since $G$ contains no induced $F$, $x_{k+2}'$ is not adjacent to $x_{i+2}$. Now let $p$ and $q$ be the first indices such that $i+2\leq p\leq n$, $k+2\leq q\leq h$, and $x_p$ is adjacent to $x_q'$ or $x_p=x_q'$. Clearly $G[W[x_{i+1},x_p]\cup W'[x_{k+1}',x_q']]$ is a hole, a contradiction.

    \textbf{Case 2.2.2.3.} Assume $x_{k+1}'$ is adjacent to $x_{i+2}$ but it is not adjacent to $x_{i+1}$. Then $G[\left\{x_i,x_{i+1},x_{i+2},x_{k-1}',x_k',x_{k+1}' \right\}]\cong D$, a contradiction.

    \textbf{Case 2.2.2.4.} Finally, assume that $x_{k+1}'$ is not adjacent to $x_{i+1}$ or $x_{i+2}$. Hence $x_{k+2}'\neq x_{i+1},x_{i+2}$. If $x_{k+2}'$ is adjacent to $x_{i+1}$, then $x_{k+2}'$ must be adjacent to $x_{i}$ since $G$ contains no induced $D$. But now $G[\left\{x_i,x_{i+1},x_{k-1}',x_k',x_{k+1}',x_{k+2}' \right\}]\cong X_5$, a contradiction. Hence $x_{k+2}'$ is not adjacent to $x_{i+1}$.

    If $x_{k+2}'$ is not adjacent to $x_{i}$, then it is obvious that there exists an induced hole, a contradiction. Hence $x_{k+2}'$ is adjacent to $x_{i}$. But now $W'$ is not a $uv$-shortest path which is impossible.

    Hence $G$ must contain at least one of $\left\{\text{hole},D,X_5,F \right\}$ as an induced subgraph. Therefore we obtain $\left\{\text{hole}, D, X_5,F\right\}$-free $\subseteq \mathbf{m_3}/\mathbf{SP}$. 
\end{proof}


Note that holes, dominoes ($D$), $X_5$, and $F$ do not belong to $\mathbf{m_3}/\mathbf{SP}$, which explains their exclusion from the sufficient condition. A natural question is whether the converse holds: must every graph in $\mathbf{m_3}/\mathbf{SP}$ be $\{\text{hole}, D, X_5, F\}$-free? Unlike most graph classes characterized by walk domination, this converse is false. Consequently, $\mathbf{m_3}/\mathbf{SP}$ is not hereditary (i.e., not closed under induced subgraphs), making a complete characterization via forbidden induced subgraphs impossible. To see this, consider the domino $D$ in Figure \ref{fig:graph}(d) and the graph $D'$ in Figure \ref{counter}. Although $D$ is an induced subgraph of $D'$, we have $D' \in \mathbf{m_3}/\mathbf{SP}$ while $D \notin \mathbf{m_3}/\mathbf{SP}$.

\begin{figure}[htbp]
\centering
\begin{tikzpicture}[
    scale=1.2,
    vertex/.style={circle, fill=black, inner sep=1.6pt},
    edge/.style={black, thick}
]
\node[vertex] (a) at (0,2) [label=above:$a$] {};
\node[vertex] (b) at (2,2) [label=above:$b$] {};
\node[vertex] (c) at (4,2) [label=above:$c$] {};

\node[vertex] (d) at (0,0) [label=below:$d$] {};
\node[vertex] (e) at (2,0) [label=below:$e$] {};
\node[vertex] (f) at (4,0) [label=below:$f$] {};

\node[vertex] (g) at (1,1) [label=left:$g$] {};
\draw[edge] (a) -- (b);
\draw[edge] (b) -- (c);
\draw[edge] (a) -- (d);
\draw[edge] (c) -- (f);
\draw[edge] (d) -- (e);
\draw[edge] (e) -- (f);

\draw[edge] (a) -- (g);
\draw[edge] (d) -- (g);
\draw[edge] (g) -- (b);
\draw[edge] (g) -- (e);
\draw[edge] (b) -- (e);
\draw[edge] (g) -- (c);
\draw[edge] (g) -- (f);
\end{tikzpicture}
\caption{An example graph $D'$ demonstrating that the converse of Theorem \ref{thm:suff} does not hold.}
\label{counter}
\end{figure}

Although the converse of Theorem \ref{thm:suff} does not hold, we can still establish the following necessary condition for graphs in $\mathbf{m_3}/\mathbf{SP}$.

\begin{theorem}\label{thm:necess}
    If $G \in \mathbf{m_3}/\mathbf{SP}$, then $G$ is hole-free, and for every induced subgraph $H$ of $G$ isomorphic to $D, X_5$, or $F$, every pair of vertices in $H$ is at a distance of at most two in $G$.
\end{theorem}
\begin{proof}
Suppose $G\in \mathbf{m_3}/\mathbf{SP}$. As shown in Figure~\ref{fig:hole}, a hole has a pair of non-adjacent vertices $u, v$ and a $uv$-shortest path: $u = x_0, x_1, x_2 = v$ which is not dominated by the $uv$-$m_3$ path: $u = x_0,x_n, x_{n-1},\ldots, x_2 = v$ ($x_1$ in the $uv$-shortest path is not dominated). Hence $G$ must be hole-$\text{free}$.
Next, we consider the cases where $G$ contains at least one of $\{D, X_5, F\}$ as an induced subgraph $H$. 

If $H \cong D$, as shown in Figure~\ref{fig:domino}, we assume for a contradiction that there exists a pair of vertices with a distance of three in $H$ that remains at distance three in $G$. By symmetry, we may assume this pair is $u = x_0$ and $v = x_3$. Then $H$ has a $uv$-shortest path $W': u = x_0, x_5, x_4, x_3 = v$, which is not dominated by the $uv$-$m_3$ path $W: u = x_0, x_1, x_2, x_3 = v$ ($x_5$ in the $uv$-shortest path is not dominated), a contradiction.

If $H \cong X_5$, as shown in Figure~\ref{fig:x5}, since $(x_0, x_3)$ is the unique pair at distance three in $X_5$, we assume for a contradiction that the distance between $x_0$ and $x_3$ remains three in $G$. Then there exists a $uv$-shortest path $W': u = x_0, x_5, x_4, x_3 = v$ which is not dominated by the $uv$-$m_3$ path $W: u = x_0, x_1, x_2, x_3 = v$ ($x_4$ in the $uv$-shortest path is not dominated), a contradiction.

If $H \cong F$, as shown in Figure~\ref{fig:f}, we assume for a contradiction that there exists a pair of vertices with a distance of three in $H$ that remains at distance three in $G$. Without loss of generality, we may assume this pair is $u = x_1$ and $v = x_4$. Then there is a $uv$-shortest path $W': u = x_1, x_2, x_3, x_4 = v$ which is not dominated by the $uv$-$m_3$ path $W: u = x_1, x_0, x_6, x_5, x_4 = v$ ($x_3$ in the $uv$-shortest path is not dominated), a contradiction.
%
%
%
\end{proof}

The converse of Theorem \ref{thm:necess} does not hold. To demonstrate this, consider the graph $G$ illustrated in Figure~\ref{fig:counter_necess}. We can easily verify that $G$ satisfies all the necessary conditions described in Theorem \ref{thm:necess}. Note that the vertex subset $S = \{a, b, c, d, f, g\}$ induces a domino $D$, and any pair of vertices within $S$ has a distance of at most two in $G$.
However, $G \notin \mathbf{m_3}/\mathbf{SP}$. Consider the vertices $a$ and $e$ with $d_G(a,e) = 3$. The $ae$-shortest path $P'': a,g,f,e$ is not dominated by the $ae$-$\mathbf{m_3}$ path $P': a,b,c,d,e$ ($g$ in the $ae$-shortest path is not dominated).

\begin{figure}[htbp]
\centering
\begin{tikzpicture}[
    scale=1.2,
    vertex/.style={circle, fill=black, inner sep=1.6pt},
    edge/.style={black, thick}
]
\node[vertex] (a) at (0,0) [label=left:$a$] {};
\node[vertex] (b) at (1,1) [label=above:$b$] {};
\node[vertex] (c) at (2,1) [label=above:$c$] {};
\node[vertex] (d) at (3,1) [label=above:$d$] {};
\node[vertex] (e) at (4,0) [label=right:$e$] {};
\node[vertex] (f) at (2.5,-1) [label=below:$f$] {};
\node[vertex] (g) at (1.5,-1) [label=below:$g$] {};
\node[vertex] (h) at (1,0) [label=below:$h$] {};

\draw[edge] (a) -- (b);
\draw[edge] (b) -- (c);
\draw[edge] (c) -- (d);
\draw[edge] (d) -- (e);
\draw[edge] (e) -- (f);
\draw[edge] (a) -- (g);

\draw[edge] (a) -- (h);
\draw[edge] (d) -- (h);
\draw[edge] (h) -- (b);
\draw[edge] (g) -- (h);
\draw[edge] (b) -- (f);
\draw[edge] (h) -- (c);
\draw[edge] (g) -- (f);
\draw[edge] (h) -- (f);
\draw[edge] (d) -- (f);
\end{tikzpicture}
\caption{A counterexample graph demonstrating that the converse of Theorem \ref{thm:necess} does not hold.}
\label{fig:counter_necess}
\end{figure}

Finally, we consider the characterizations of the class $\mathbf{m_3}/\mathbf{l_k}$ for $k=2,3$. By Remark \ref{rem1}, $\mathbf{m_3}/\mathbf{IP}\subseteq \mathbf{m_3}/\mathbf{l_3}\subseteq \mathbf{m_3}/\mathbf{l_2}$. The case of $\mathbf{m_3}/\mathbf{l_2}$ is much easier.
\begin{theorem}\label{l2}
    $\mathbf{m_3}/\mathbf{l_2}=\text{hole-free}$.
\end{theorem}
\begin{proof}
    We first show that a hole is not in $\mathbf{m_3}/\mathbf{l_2}$. As shown in Figure \ref{fig:hole}, a hole has a pair of non-adjacent vertices $u,v$ and a $uv$-$l_3$-path $u=x_0,x_1,x_2=v$ which is not dominated by the $uv$-$m_3$ path $u=x_0,x_n,x_{n-1},\ldots,x_2=v$ ($x_1$ in the $uv$-$l_2$-path is not dominated). Hence, $\mathbf{m_3}/\mathbf{l_2}\subseteq \text{hole-free}$.

     On the other hand, let $G$ be a hole-free graph. In order to derive a contradiction, suppose $G\not\in \mathbf{m_3}/\mathbf{l_2}$. Then there exist two non-adjacent vertices $u$ and $v$, a $uv$-$m_3$ path $W: u=x_0,\ldots,x_n=v$ $(n\ge3)$ and a $uv$-$l_2$-path $W'$ satisfying that $W$ does not dominate $W'$. Thus, there is some internal vertex of $W'$ that is neither a vertex of $W$ nor adjacent to any internal vertex of $W$. Note that the length of $W'$ must be two since $u$ is not adjacent to $v$. Hence $W':u=x_0',x_1,x_2'=v$ and $x_2'$ is not adjacent to any internal vertex in $W$. Since $W$ is an $m_3$ path, there must exist a hole in $G$, a contradiction. Hence, hole-free $\subseteq \mathbf{m_3}/\mathbf{l_2}$. 
\end{proof}
\begin{theorem}
    $\mathbf{m_3}/\mathbf{l_3}=\left\{\text{hole},D, F,X_5\right\}$-$\text{free}$.
\end{theorem}
\begin{proof}
    Analogous to the preceding proof, we first show that the forbidden graphs do not belong to the class $\mathbf{m_3}/\mathbf{l_3}$.
    
    As shown in Figure \ref{fig:hole}, a hole has a pair of non-adjacent vertices $u,v$ and a $uv$-$l_3$-path $u=x_0,x_1,x_2=v$ which is not dominated by the $uv$-$m_3$ path $u=x_0,x_n,x_{n-1},\ldots,x_2=v$ ($x_1$ in the $uv$-$l_3$-path is not dominated).
    
    As shown in Figure \ref{fig:domino}, $D$ has a pair of non-adjacent vertices $u,v$ and a $uv$-$l_3$-path $u=x_0,x_1,x_2,x_3=v$ which is not dominated by the $uv$-$m_3$ path $u=x_0,x_5,x_4,x_3=v$ ($x_2$ in the $uv$-$l_3$-path is not dominated). 
     
    As shown in Figure \ref{fig:f}, $F$ has a pair of non-adjacent vertices $u,v$ and a $uv$-$l_3$-path $u=x_1,x_2,x_3,x_4=v$ which is not dominated by the $uv$-$m_3$ path $u=x_1,x_0,x_6,x_5,x_4=v$ ($x_3$ in the $uv$-$l_3$-path is not dominated). 
    
    As shown in Figure \ref{fig:x5}, $X_5$ has a pair of non-adjacent vertices $u,v$ and a $uv$-$l_3$-path $u=x_0,x_5,x_4,x_3=v$ which is not dominated by the $uv$-$m_3$ path $u=x_0,x_1,x_2,x_3=v$ ($x_4$ in the $uv$-$l_3$-path is not dominated).


    Now we prove that $\left\{\text{hole},D, F, X_5 \right\}$-free $\subseteq \mathbf{m_3}/\mathbf{l_3}$. 
    Let $G$ be a $\left\{\text{hole},D,F, X_5 \right\}$-free graph.
    In order to derive a contradiction, suppose $G\not\in \mathbf{m_3}/\mathbf{l_3}$. Then there exist two non-adjacent vertices $u$ and $v$, a $uv$-$m_3$ path $W: u=x_0,\ldots,x_n=v$ $(n\ge3)$ and a $uv$-$l_3$-path $W'$ satisfying that $W$ does not dominate $W'$. Thus, there is some internal vertex of $W'$ that is neither a vertex of $W$ nor adjacent to any internal vertex of $W$. Note that the length of $W'$ must be at least two since $u$ is not adjacent to $v$. Then by Theorem~\ref{l2}, we can suppose the length of $W'$ is three and hence $W':u=x_0',x_1',x_2',x_3'=v$. Observe that $x_1',x_2'\not\in W$. Also, $x_1'$ and $x_2'$ cannot both be adjacent to internal vertices in $W$. If $x_1'$ and $x_2'$ are neither adjacent to internal vertices in $W$, then $G[W\cup W']$ is a hole, a contradiction.

    Hence we suppose that $x_1'$ is adjacent to some internal vertex in $W$ but $x_2'$ is not adjacent to any internal vertex in $W$. Let $i$ be the last index such that $x_1'$ is adjacent to $x_i$. We assert that $i=n-1$ since otherwise $G[W[x_i,v]\cup W'[x_1',v]]$ is a hole. 

    If $x_{i-1}$ is not adjacent to $x_1'$ and $u$, we suppose $x_j$ is the vertex in $W$ which is adjacent to $x_1'$ before $x_i$. Then either $G[W[x_j,v]\cup W'[x_1',v]]\cong D$ or $G[W[x_j,x_i]\cup \left\{x_1'\right\}]$ is a hole, a contradiction. If $x_{i-1}$ is adjacent to both $u$ and $x_1'$, then we have $G[\left\{x_i,x_1',x_2',v,u,x_{i-1} \right\}]\cong X_5$, a contradiction. Hence $x_{i-1}$ is only adjacent to one of $u$ and $x_1'$.

    \textbf{Case 1.} Suppose $x_{i-1}$ is only adjacent to $x_1'$, then $G[\left\{x_i,x_1',x_2',v,x_{i-1} \right\}]$ is a house. 
    Note that $x_{i-2}$ cannot be adjacent to $x_1'$ since otherwise $G[\{x_i,x_1',x_2',v,x_{i-1}$, $x_{i-2} \}]\cong X_5$, a contradiction. 
    Hence there exists $x_{i-3}\in W$ which may be equal to $u$.
    But now we have $G[\left\{x_i,x_1',x_2',v,x_{i-1},x_{i-2},x_{i-3} \right\}]\cong F$ or $G[\{x_1' \}\cup W[u,x_{i-1}]]$ contains a hole as an induced subgraph, a contradiction.

    \textbf{Case 2.} Suppose $x_{i-1}$ is only adjacent to $u$, now $G[\left\{x_i,x_1',x_2',v,u,x_{i-1} \right\}]\cong D$, a contradiction.

    When $x_2'$ is adjacent to some internal vertex in $W$ but $x_1'$ is not adjacent to any internal vertex in $W$, the proof is similar. Hence, $\left\{\text{hole}, D,F, X_5 \right\}$-free $\subseteq \mathbf{m_3}/\mathbf{l_3}$. Therefore,  $\mathbf{m_3}/\mathbf{l_3}=\left\{\text{hole},D, F, X_5 \right\}$-free. 
\end{proof}

\section{Conclusions}\label{Con}

In this paper, we continue the study of walk domination between $m_3$-paths and various other types of walks. First, we characterize the classes $\mathbf{A}/\mathbf{m_3}$ for $\mathbf{A}\in \left\{\mathbf{l_k},\mathbf{SP},\mathbf{TW},\mathbf{WTW} \right\}$ with $k=2,3$, proving that all these classes coincide with the class of $\text{HHD-free}$ graphs. 

Second, we consider the classes $\mathbf{m_3}/\mathbf{A}$ for $\mathbf{A}\in \left\{\mathbf{l_k},\mathbf{SP} \right\}$ with $k=2,3$. By introducing the graph $F$ as a new forbidden subgraph, we establish that $\mathbf{m_3}/\mathbf{l_3}=\left\{\text{hole}, D, X_5,F\right\}\text{-free}$ and $\mathbf{m_3}/\mathbf{l_2}=\text{hole-free}$. For the class $\mathbf{m_3}/\mathbf{SP}$, although we demonstrate that a complete characterization via forbidden induced subgraphs is impossible due to its non-hereditary nature, we still establish being $\left\{\text{hole}, D, X_5, F\right\}\text{-free}$ as a sufficient condition. Finding its exact characterization remains an open problem.

Furthermore, precise characterizations for the classes $\mathbf{m_3}/\mathbf{WTW}$ and $\mathbf{m_3}/\mathbf{TW}$ also remain open.
The reason is that these walk types allow for chords, which complicates the analysis.
The existence of such chords leads to a greater number of potential forbidden subgraphs, making the characterization significantly more challenging.


\paragraph{Data availability} Not applicable.

\section*{Declaration}

\paragraph{Conflict of interest}
The authors declare that they have no conflict of interest.

\section*{Acknowledgements}
The authors are highly grateful to the anonymous referee for pointing out that the converse of Theorem \ref{thm:suff} does not hold, which revealed the non-hereditary nature of the graph class $\mathbf{m_3}/\mathbf{SP}$.


\end{document}